\documentclass[11pt,leqno]{amsart}

\usepackage{amscd}
\usepackage{color}
\usepackage{amssymb}
\usepackage{amsfonts}
\usepackage{latexsym}
\usepackage{verbatim}
\usepackage{enumerate}

\theoremstyle{plain}
\newtheorem{theorem}{Theorem}[section]

\newtheorem{lemma}[theorem]{Lemma}

\newtheorem{rem}[theorem]{Remark}

\renewcommand{\b}{\begin{equation}}
\newcommand{\e}{\end{equation}}

\newcommand\C{{\mathbb C}}
\newcommand\R{{\mathbb R}}

\title{On the stability of the anomaly flow}
\thanks{This work was supported by GNSAGA of INdAM}
\address{Dipartimento di Ingegneria e Scienze dell'Informazione e Matematica \\ Universit\`a dell'Aquila\\
via Vetoio\\ 67100 L'Aquila\\ Italy}
\email{lucio.bedulli@univaq.it}
\address{Dipartimento di Matematica G. Peano \\ Universit\`a di Torino\\
Via Carlo Alberto 10\\
10123 Torino\\ Italy}
 \email{luigi.vezzoni@unito.it}

\author{Lucio Bedulli and Luigi Vezzoni}

\begin{document}

\maketitle
\begin{abstract}
We prove that the parabolic flow of conformally balanced metrics introduced in \cite{AF7} is stable around Calabi-Yau metrics. The result shows that the flow can converge  on a K\"ahler manifold even if the initial metric is not conformally K\"ahler. 
\end{abstract}

\section{Introduction}

Anomaly flow is a geometric flow of  Hermitian metrics studied in \cite{AF1,AF9,AF2,AF8,AF3,AF4,AF5,AF6,AF7,UgartePujia}. The flow was originally considered in \cite{AF8} on complex threefolds to study the Strominger system \cite{Strominger} and involves a real parameter $\alpha'$. The flow was subsequently generalized to any complex dimension $n \geq 3$ for $\alpha'=0$ in \cite{AF7}.  The latter evolves an initial Hermitian metric $\omega_0$ on a compact complex manifolds $M$
of complex dimension $n\geq 3$ with $c_1(M)=0$ by  
\begin{equation}\label{phongflow}
\tfrac{\partial}{\partial t}(|\Omega|_{\omega}\omega^{n-1})=i\partial\bar\partial \omega^{n-2}\,,\quad \omega(0)=\omega_0\,,
\end{equation}
where $\Omega$ is a fixed complex volume form and $|\cdot|_{\omega}$ is the pointwise norm with respect to $\omega$. By \lq\lq complex volume form\rq\rq $\,$ we just mean a nowhere vanishing $(n,0)$-form, indeed for the purpose of the present paper we do not need to assume $\Omega$ to be holomorphic.

The well-posedness of the flow is proved in \cite[Theorem 1]{AF7}  under the assumption on $\omega_0$ to be conformally balanced (in such a case the components of $\omega(t)$ satisfy a parabolic system \cite[Theorem 4]{AF7}). Moreover, when $\omega_0$ is conformally balanced, \eqref{phongflow} is conformally equivalent to the Hermitian curvature flow introduced by Ustinovskiy in \cite{YU1} (see \cite{AF9}). 

The flow \eqref{phongflow} can only converge when $M$ is K\"ahler. The research of the present paper is motivated by the following theorem about the long time existence and convergence of the flow when  $|\Omega|_{\omega_0}^{1/(n-1)}\omega_0$ is a K\"ahler metric:

\begin{theorem}[Phong, Picard and  Zhang  {\cite[Theorem 2]{AF7}}]\label{phong}
Let $(M, \chi)$ be a compact K\"ahler manifold with vanishing first Chern class and let $\Omega$ be a complex volume form on  $M$ with constant norm with respect to $\chi$. Let $\omega_0$ be a Hermitian metric on $M$ such that 
$$
|\Omega|_{\omega_0}\omega^{n-1}_0=\chi^{n-1}\,,
$$
then \eqref{phongflow} starting from $\omega_0$ has a long-time solution which converges in $C^{\infty}$--topology to the unique Ricci-flat  K\"ahler metric $\omega_{\infty}\in [\chi] \in H^{1,1}(M)$ .
\end{theorem}
The theorem gives an alternative proof of the Calabi-Yau theorem \cite{yau}. In \cite{AF7} it is raised the problem of studying the convergence of the flow in K\"ahler manifolds when $|\Omega|_{\omega_0}^{1/(n-1)}\omega_0$ is not K\"ahler, for instance when $|\Omega|_{\omega_0}\omega^{n-1}_0$ is just closed and $[|\Omega|_{\omega_0}\omega^{n-1}_0]=[\chi^{n-1}]$, with $\chi$ K\"ahler.  
Here we prove that \eqref{phongflow} is stable around Calabi-Yau metrics $\chi$ with $|\Omega|_{\chi}$ constant. In particular we have the convergence of \eqref{phongflow}, when $|\Omega|_{\omega_0}\omega^{n-1}_0$ is just close to the $(n-1)$-th power of a K\"ahler metric.
%not the $(n-1)$-th power of a K\"ahler metric $\chi$, but just close to $\chi$. 

\begin{theorem}\label{main}
Let $(M, \chi)$ be a compact K\"ahler manifold with vanishing first Chern class and let $\Omega$ be a complex volume form on  $M$ with constant norm with respect to $\chi$. For every $\epsilon>0$ there exists $\delta>0$ such that if $\omega_0$ is a Hermitian metric on $M$ satisfying 
\begin{equation}\label{vicino}
\|\,|\Omega|_{\omega_0}\omega^{n-1}_0-\chi^{n-1}\|_{C^{\infty}}<\delta\,, 
\end{equation}
then flow \eqref{phongflow} has a long-time solution $\omega(t)\in C^{\infty}(M\times [0,\infty),\Lambda_{+}^{1,1})$ such that 
$$
\|\,|\Omega|_{\omega(t)}\omega(t)^{n-1}-\chi^{n-1}\|_{C^{\infty}}<\epsilon\,,\mbox{ for every }t\in [0,\infty)
$$
and $|\Omega|_{\omega(t)}\omega(t)^{n-1}$ converges in $C^{\infty}$--topology to a positive $(n\!-\!1,n\!-\!1)$-form $\omega_{\infty}^{n-1}$ with $\omega_{\infty}$ astheno-K\"ahler. 

If further $\omega_0$ satisifies the conformally balanced condition $d(|\Omega|_{\omega_0}\omega_{0}^{n-1})=0$, then $\omega_\infty$ is K\"ahler Ricci flat.
\end{theorem}

We prove Theorem \ref{main} as follows: 

\smallskip 
after the change of variable $|\Omega|_{\omega}\omega^{n-1} = \tilde\omega^{n-1}$, equation \eqref{phongflow} rewrites as 
\begin{equation}
\label{ourflow}
\tfrac{\partial}{\partial t} \tilde\omega^{n-1}=i\partial\bar\partial (|\Omega|_{ \tilde\omega}^{-2} \tilde\omega^{n-2})\,,
\end{equation}
(see Lemma \ref{lemma31}). 

If we set $E(\tilde\omega^{n-1})=i\partial\bar\partial (|\Omega|_{ \tilde\omega}^{-2} \tilde\omega^{n-2})$, then the linearization $DE(\alpha^{n-1})$ of $E$ at any positive $(n-1,n-1)$-real form $\alpha^{n-1}$ satisfies 
$$
DE(\alpha^{n-1})(\psi)=-\frac{1}{n-1}|\Omega|_{\alpha}\square \psi+{\rm l.o.t.}
$$ 
for any {\em closed} $\psi\in C^{\infty}(M,\Lambda^{n-1,n-1}_\R)$ (see Lemma \ref{lin_E}) and the flow \eqref{ourflow} fits in Hamilton's framework \cite[Section 5 and 6]{positive} with integrability condition $L=d$ (this is analogous to the argument used in \cite{AF7}). That in particular implies the well-posedness of the flow \eqref{phongflow} also when $\omega_0$ is not conformally balanced.   

Moreover if $\chi$ is K\"ahler, then 
$$
DE(\chi^{n-1})(\psi)=-\frac{1}{n-1}|\Omega|_{\omega}\square \psi
$$
for any closed $\psi\in C^{\infty}(M,\Lambda^{n-1,n-1}_\R)$. This allows us to apply a general result about the stability of second order geometric flows with an integrability condition. We state this theorem in section \ref{description} and we prove it in the last section.  

\bigskip 
\noindent{\em Notation.} Given a vector bundle $F$ on a manifold $M$, we denote by $C^{\infty}(M,F)$ the space of smooth sections of $F$. If further $I\subset \R$ is an interval we denote by $C^{\infty}(M\times I,F)$
the space of smooth time depending sections of $F$. When we write $\|f\|_{C^{\infty}}<\delta$, we mean that $\|f\|_{C^{k}}<\delta$ for every $k\in \mathbb N$. 

\section{A stability result for  second order geometric flows with an integrability condition} \label{description}
In \cite{positive} Hamilton proved the following general result about the short-time existence of second order geometric flows on compact manifolds.

\medskip 
Let $M$ be an oriented compact manifold, $F$ a vector bundle over $M$, $U$ an open subbundle of $F$ and  
$$
E\colon C^{\infty}(M,U)\to  C^{\infty}(M,F)
$$
a second order differential operator. Consider the geometric flow 
\begin{equation}\label{flow_Ham}
\frac{\partial f}{\partial t}=E(f)\,,\quad f(0)=f_0\,,
\end{equation}
where $f_0$ belongs to $C^{\infty}(M,U)$. 
For $f\in C^{\infty}(M,U)$, we denote by $D E(f)\colon C^{\infty}(M,F)\to C^{\infty}(M,F)$ the linearization of $E$ at $f$ and by 
$\sigma D E(f)$ the principal symbol of
$D E(f)$. Following Hamilton's paper we assume that there exists a first order linear differential operator 
$$ 
L\colon C^{\infty}(M,F)\to C^{\infty}(M,G)\,,
$$
with values  in another vector bundle $G$ over $M$, such that 
\begin{enumerate}
\item[1.] $L(E(f))=0$ for all $f\in C^{\infty}(M,U)$;

\vspace{0.1cm}
\item[2.] for every $f\in C^{\infty}(M,U)$ and for every $(x,\xi) \in T^*M$ all the eigenvalues of $\sigma DE(f)(x,\xi)$ restricted to $\ker \sigma L(x, \xi)$
 have strictly positive real part\,.
 \end{enumerate}
Because of the following result $L$ is called an {\em integrability condition for $E$}.

\begin{theorem}[Hamilton {\cite[Theorem 5.1]{positive}} ]\label{Ham_int}
Under the above  assumptions the geometric flow \eqref{flow_Ham} has a unique short-time solution.   
\end{theorem}

\begin{rem} {\em Theorem 5.1 in \cite{positive} is in fact more general since the integrability condition $L$ is allowed to smoothly depend on $f \in C^\infty(M,U)$. This generality is needed to prove the short time existence of the Ricci flow.}
\end{rem}

Using Theorem \ref{Ham_int} we will be able to prove the following stability theorem for geometric flows with an integrability condition $L$.

\begin{theorem}\label{stabilityham}
Assume that $E$ and $L$ are as above.
Let $\bar f\in C^{\infty}(M,U)$ be such that $E(\bar f)=0$. Let $\bar h$ be a fixed metric along the fibers of $F$. Assume
\begin{enumerate}[(i)]
\item \label{semidef} $D E(\bar f)\colon \ker L \to \ker L$ is symmetric and negative semidefinite with respect to $\bar h$;
\item \label{miracolo} $E(f)$ is $L^2$-orthogonal to $\ker D E(\bar f)$ for every $f \in C^\infty(M,U)$;
%\item $DE(\bar f): C^\infty(M,F) \to C^\infty(M,F)$ is elliptic.
%\item There exists an elliptic operator $\Phi: C^\infty(M,F) \to C^\infty(M,F)$ such that $\Phi_{|\ker L}=DE(\bar f)_{|\ker L}$.
%\color{red}{\item $D E(\bar f)\colon \ker L \to \ker L$ extends to an elliptic operator $\Phi: C^\infty(M,F) \to C^\infty(M,F)$. }
\item $D E(\bar f)\colon \ker L \to \ker L$ extends to an elliptic operator $\Phi: C^\infty(M,F) \to C^\infty(M,F)$.
%\item $D E(\bar f)\colon \ker L \to \ker L$ is the restriction to $\ker L$ of an elliptic operator $\Phi: C^\infty(M,F) \to C^\infty(M,F)$.
\end{enumerate}
Then for every $\epsilon>0$ there exists $\delta>0$ such that if 
$f_0\in C^{\infty}(M,U)$ satisfies 
$$
\|f_0-\bar f\|_{C^\infty}<\delta, 
$$
then \eqref{flow_Ham} has a long-time solution $f\in C^{\infty}(M\times [0,\infty),U)$ such that 
$$
\|f(t)-\bar f\|_{C^	\infty}<\epsilon, \mbox{ for every $t\in[0,\infty)$}.  
$$
Moreover, $f(t)$ converges to $f_{\infty}\in C^{\infty}(M,U)$ in $C^{\infty}$--topology which satisfies $E(f_{\infty})=0$.  
\end{theorem}

\section{Proof of Theorem \ref{main}}
Let $(M,\omega_0)$ be a compact  $n$-dimensional  Hermitian manifold with vanishing first Chern class and  let $\Omega$ be a fixed complex volume form. 
\begin{lemma}
\label{lemma31}
Let $\omega(t)$ be a solution to the geometric flow \eqref{phongflow} on $M$;  then $\tilde{\omega}=|\Omega|_{\omega}^{1/(n-1)}\omega$ satisfies 
\begin{equation}\label{phongflow2}
\tfrac{\partial}{\partial t}\tilde \omega^{n-1}=i\partial\bar\partial (|\Omega|_{\tilde \omega}^{-2}\tilde \omega^{n-2})\,.
\end{equation}
\end{lemma}
\begin{proof}
Since 
$$
\omega=|\Omega|_{\omega}^{-1/(n-1)}\tilde{\omega}\,,
$$
we have 
$$
\tfrac{\partial}{\partial t}(\tilde{\omega}^{n-1})=i\partial\bar\partial(|\Omega|_{\omega}^{-(n-2)/(n-1)} \tilde{\omega}^{n-2})\,.
$$
Now in general  for any conformal factor $f\in C^{\infty}(M,\R_{+})$ one has
$$
|\Omega|_{f\tilde{\omega}}=f^{-n/2} |\Omega|_{\tilde{\omega}}
$$
and thus
$$
 |\Omega|_{\omega}=(|\Omega|_{\omega}^{-1/(n-1)})^{-n/2} |\Omega|_{\tilde{\omega}}=
 |\Omega|_{\omega}^{n/(2n-2)} |\Omega|_{\tilde{\omega}}
$$
from which we deduce 
$$
|\Omega|_{\omega}= |\Omega|_{\tilde{\omega}}^{(2n-2)/(n-2)}\,,
$$
and the claim follows.
\end{proof}

Now we focus on the geometric flow \eqref{phongflow2} and we show that it fits in the set-up of Theorem \ref{stabilityham}. The flow is governed by the operator 
$$
E\colon C^{\infty}(M,\Lambda^{n-1,n-1}_{+})\to C^{\infty}(M,\Lambda^{n-1,n-1}_{\R})
$$
defined by 
\begin{equation}\label{EEE}
E(\omega^{n-1})=i\partial \bar \partial (|\Omega|_{\omega}^{-2}\,\omega^{n-2})
\end{equation}
where $\Lambda^{n-1,n-1}_{+}$ is the bundle of positive real $(n\!-\!1,n\!-\!1)$-forms on $M$ and $\Lambda^{n-1,n-1}_{\R}$ is the bundle of real $(n\!-\!1,n\!-\!1)$-forms.

\medskip 
In order to study the linearization of $E$, we describe the principal part of the operator $\square$ in terms of the components of $(n\!-\!1,n\!-\!1)$-real forms on $M$. 

Let $\omega$ be any Hermitian metric on $M$ and $\psi\in C^{\infty}(M,\Lambda_{\R}^{n-1,n-1})$. Then $\psi$ writes in a unique way as
\begin{equation}\label{psi}
\psi=\frac{1}{(n-1)!}h_0\omega^{n-1}-\frac{1}{(n-2)!}h_2\wedge\omega^{n-2}\,,
\end{equation}
where $h_0$ is a smooth function and $h_2\in C^{\infty}(M,\Lambda^{1,1}_{\R})$ satisfies 
$$
h_2	\wedge \omega^{n-1}=0\,. 
$$
Since
$$
*h_2=-\frac{1}{(n-2)!}h_2\wedge\omega^{n-2}\,,
$$
the form $\psi$ can be alternatively written as 
$$
\psi =*(h_0\omega +h_2)\,. 
$$

\begin{lemma}\label{lap}
If $\psi$ is closed, then  
$$
\square\,\psi=\partial \partial^*\psi=-\frac{2i}{(n-2)!}\partial \bar\partial h_0\wedge \omega^{n-2}+\frac{i}{(n-3)!}\partial\bar\partial h_2\wedge\omega^{n-3}+{\rm l.o.t.}
$$
where {\rm \lq\lq l.o.t.\rq\rq} stands for \lq\lq lower order terms\rq\rq\, in  $\psi$. Moreover if $\omega$ is K\"ahler we have 
$$
\square\,\psi=\partial \partial^*\psi=-\frac{2i}{(n-2)!}\partial \bar\partial h_0\wedge \omega^{n-2}+\frac{i}{(n-3)!}\partial\bar\partial h_2\wedge\omega^{n-3}\,.
$$
\end{lemma}
\begin{proof}
Since $\psi$ is closed we have 
$$
\square\,\psi=\partial \partial^*\psi =-\partial *\bar\partial *\psi =
-\partial *(\bar\partial h_0\wedge \omega)-\partial *\bar\partial h_2 +{\rm l.o.t.}
$$  
On the other hand, from the closure of $\psi$, we deduce 
$$
\bar \partial h_2\wedge\omega^{n-2}=\frac{1}{n-1}\bar \partial h_0\wedge \omega^{n-1}+{\rm l.o.t}\,. 
$$
Now we use the well-known splitting of $(1,2)$-forms as
$$\gamma = \gamma_+ + \gamma_-$$
where $\gamma_+$ is of the form $\alpha \wedge \omega$ with $\alpha$ a $(0,1)$-form and $\gamma_-$ is such that $\gamma_- \wedge \omega^{n-2}=0$. \\
The fact we use is that $* \gamma_+ = \frac{i}{(n-2)!} \gamma_+ \wedge \omega^{n-3}$ and $*\gamma_-=- \frac{i}{(n-3)!} \gamma_- \wedge \omega^{n-3}$.
Thus in our case, taking into account that $(\bar\partial h_2)_+=\frac{1}{n-1} \bar\partial h_0 \wedge \omega +{\rm l.o.t.}$,
we have
$$
\begin{aligned}
*\bar\partial h_2=&\frac{i}{(n-1)!}\bar \partial h_0\wedge \omega^{n-2}-\frac{i}{(n-3)!}\left(\bar\partial h_2-\frac{1}{n-1}\bar \partial h_0\wedge \omega\right)\wedge\omega^{n-3}+{\rm l.o.t.}\\
&=\frac{i}{(n-2)!}\bar \partial h_0\wedge \omega^{n-2}-\frac{i}{(n-3)!}\bar\partial h_2\wedge\omega^{n-3}+{\rm l.o.t.}
\end{aligned}
$$
Therefore 
$$
\square\,\psi=-\frac{2i}{(n-2)!}\partial \bar\partial h_0\wedge \omega^{n-2}+\frac{i}{(n-3)!}\partial\bar\partial h_2\wedge\omega^{n-3}+{\rm l.o.t.},
$$
where these lower order terms vanish if $\omega$ is closed since they all come from $d\omega$. \end{proof}

\begin{lemma}
\label{lin_E}
Let $\omega$ be a Hermitian metric on $M$ and let $\psi \in C^{\infty}(M,\Lambda^{n-1,n-1}_{\R})$ be closed; then 
\begin{equation}\label{DEnonK}
DE(\omega^{n-1})(\psi)=-\frac{1}{n-1}|\Omega|_{\omega}\square \psi+{\rm l.o.t.}
\end{equation}
Moreover if we assume that $\omega$ is K\"ahler then we have 
$$
DE(\omega^{n-1})(\psi)=-\frac{1}{n-1}|\Omega|_{\omega}\square \psi\,. 
$$ 
\end{lemma}

\begin{proof}
Let $\omega(t)$, $t\in (-\epsilon,\epsilon)$ be a smooth curve of Hermitian metrics with $\omega(0)=\omega$ on $M$ and we assume that 
$$
\psi:=\tfrac{\partial}{\partial t}_{|t=0}\omega^{n-1}(t)
$$
is closed. In order to simplify the notation  we set 
$$
r(t)=|\Omega|_{\omega(t)}^{-2}\,,\quad r=r(0)\,,\quad \dot r=\tfrac{\partial}{\partial t}_{|t=0}r(t)\,,\quad \dot \omega=\tfrac{\partial}{\partial t}_{|t=0}\omega(t)\,.
$$

We directly compute 
$$
\begin{aligned}
\tfrac{\partial}{\partial t}_{|t=0}E(\omega(t)^{n-1})=&\,i\partial \bar \partial (\dot{r}\,\omega^{n-2}+(n-2) r\, \dot{\omega}\wedge\omega^{n-3})\\
=&i\partial \bar \partial \dot{r}\,\wedge \omega^{n-2}+(n-2)i r\, \partial \bar \partial\dot{\omega}\wedge \omega^{n-3}+{\rm l.o.t.}
\end{aligned}
$$
We decompose $\psi$ according to \eqref{psi}.  Then \cite[Lemma 2.5]{Crelle} implies   
$$
\dot \omega=\frac{h_0}{(n-1)(n-1)!}\omega-\frac{1}{(n-1)!}h_2
$$
and using 
$$
\dot r=\frac{nh_0}{(n-1)(n-1)!} r
$$
we obtain 
\begin{multline*}
\tfrac{d}{dt}_{|t=0}E(\omega(t)^{n-1})=\frac{n}{(n-1)(n-1)!}i r\, \partial \bar \partial h_0\wedge \omega^{n-2}\\+\frac{n-2}{(n-1)(n-1)!}i r\, \partial \bar \partial h_0\wedge \omega^{n-2}-\frac{n-2}{(n-1)!}i r \,\partial \bar \partial h_2\wedge \omega^{n-3}+{\rm l.o.t.}\,,
\end{multline*}
i.e.
\begin{equation}\label{pre}
\tfrac{d}{dt}_{|t=0}E(\omega(t)^{n-1})=\frac{2}{(n-1)!}i r \,\partial \bar \partial h_0\wedge \omega^{n-2}-\frac{n-2}{(n-1)!}i r \, \partial \bar \partial h_2\wedge \omega^{n-3}+{\rm l.o.t.}
\end{equation}
and Lemma \ref{lap} implies \eqref{DEnonK}. Since the lower order terms in \eqref{pre} vanish if $\omega$  is closed, the claim follows. 
\end{proof}

\begin{proof}[Proof of Theorem $\ref{main}$]
Let $(M,\chi)$ be a compact K\"ahler manifold of complex dimension $n$. Let $\Omega$ be a complex volume form on $M$ with constant norm with respect to $\chi$ and let 
$E\colon C^{\infty}(M,\Lambda^{n-1,n-1}_{+})\to C^{\infty}(M,\Lambda^{n-1,n-1}_{\R})$ be as in \eqref{EEE}. Then we have 
$$
E(\chi^{n-1})=0\,.
$$ 
Now Lemma \ref{lin_E} and Hodge theory imply that all the assumptions of Theorem \ref{stabilityham} are satisfied when we consider
$$
\begin{array}{llll}
& F=\Lambda^{n-1,n-1}_\R\,, & U= \Lambda^{n-1,n-1}_+\,, & L=d\colon C^{\infty}(M,\Lambda^{n-1,n-1}_\R)\to C^{\infty}(M,\Lambda^{2n-1}_\C)\,.
\end{array}
$$
Hence for every $\epsilon>0$ there exists $\delta>0$ such that if $\tilde \omega_0$ is a Hermitian metric on $M$ satisfying 
\begin{equation}\label{vicinotilde}
\|\tilde \omega_0^{n-1}-\chi^{n-1}\|_{C^{\infty}}<\delta\,,
\end{equation}
then there exists  a smooth family of Hermitian metrics $\tilde \omega(t)$, $t\in [0,\infty)$, such that 
\begin{equation}\label{vicinotilde1}
\|\tilde \omega(t)^{n-1}-\chi\|_{C^{\infty}}<\epsilon\,,\quad \tfrac{\partial}{\partial t}\tilde \omega=i\partial \bar\partial(|\Omega|_{\tilde \omega}^{-2}\tilde\omega^{n-2})\,,\quad \tilde \omega(0)=\tilde \omega_0
\end{equation}
and $\tilde \omega(t)$ converges in $C^{\infty}$--topology to a Hermitian metric $\tilde\omega_{\infty}$ such that
\begin{equation}\label{vicinotilde2}
\partial \bar\partial (|\Omega|_{\tilde \omega_{\infty}}^{-2}\,\tilde \omega_{\infty}^{n-2})=0\,. 
\end{equation}

Now let $\omega_{0}$ be a Hermitian metric on $M$ satisfying \eqref{vicino}; then $\tilde\omega_0:=|\Omega|_{\omega_0}^{1/(n-1)}\omega_0$ satisfies \eqref{vicinotilde}. Thus there exists $\tilde \omega \in C^{\infty}(M\times[0,\infty),\Lambda^{1,1}_{\R})$ satisfying \eqref{vicinotilde1} and  converging in $C^{\infty}$-topology to a $\tilde \omega_\infty$ for which \eqref{vicinotilde2} holds. Therefore
$\omega=|\Omega|_{\tilde \omega}^{-2/(n-2)}\tilde{\omega}$ is a solution to the anomaly flow \eqref{phongflow} satisfying
$$
\||\Omega|_{\omega(t)} \omega(t)^{n-1}-\chi\|_{C^{\infty}}<\epsilon\,,
$$
and $|\Omega|_{\omega(t)}\omega(t)^{n-1}$ converges in $C^{\infty}$-topology to 
$$
\omega_{\infty}^{n-1}=|\Omega|_{\tilde\omega_{\infty}}^{-2(n-1)/(n-2)}\tilde{\omega}_{\infty}^{n-1}\,.
$$
By a change of variable we obtain that $\omega_{\infty}$ is astheno-K\"ahler, i.e.  
$$
\partial 	\bar\partial (\omega_{\infty}^{n-2})=0\,, 
$$
and the first part of the claim follows. 

\medskip 
About the second part of the claim, we use that if $\omega_0$ satisfies the conformally balanced condition
$d(|\Omega|_{\omega_0}\omega^{n-1}_0)=0$, then $d(|\Omega|_{\omega(t)}\omega(t)^{n-1})=0$ for every $t$
and hence $\omega_{\infty}$ satisfies 
$$
d(|\Omega|_{\omega_\infty}\omega_\infty^{n-1})=0	\,,\quad \partial 	\bar\partial (\omega_{\infty}^{n-2})=0\,. 
$$
In view of \cite[Lemma 1]{AF7} (see also \cite{Matsuo}), $\omega_\infty$ is K\"ahler-Einstein. 
\end{proof}

\section{Proof of Theorem \ref{stabilityham}}

In this section we prove the general result about the stability of geometric flows in Hamilton's set-up described in section \ref{description}. 
\begin{proof}
We adapt the proof of the main theorem in \cite{Advances}.  \\ 
Fix a metric connection $\nabla$ on $(F,\bar h)$ and a volume form on $M$. The space $C^{\infty}(M,F)$  has the natural structure of tame Fr\'echet space given by the Sobolev norms $\|.\|_{H^n}$ induced by $\bar h$, $\nabla$ and the volume form of $M$.

\medskip 
On $C^{\infty}(M\times [a,b], F)$ we consider the grading
$$
\|f\|_{n,[a,b]}=\sum_{2j\leq n}\int_{a}^{b}\|\partial_t^{j} f(t) \|_{H^{n-2j}}\, dt.
$$

Hamilton in \cite[sections 5 and 6]{positive} proved that, with respect to this grading, for any $T>0$ the map
$$
\mathcal F\colon C^{\infty}(M\times[0,T],U)\to C^{\infty}(M\times[0,T],F)\times C^{\infty}(M,F) 
$$
$$
\mathcal F(f)=(\partial f/\partial t-E(f), f(0))
$$
satisfies the assumptions of the Nash-Moser theorem, i.e. $\mathcal F$ is smooth tame, $D\mathcal F(f)$ is bijective for every 
$f\in C^{\infty}(M\times[0,T],U)$ and the family of the inverses
$$
C^{\infty}(M\times[0,T],U)\times  C^{\infty}(M\times[0,T],F)\times C^{\infty}(M,F)\to C^{\infty}(M\times[0,T],F)
$$
$$
(f,(g,k))\mapsto D \mathcal F (f)^{-1}(g,k)
$$
is a smooth tame map. Hence the Nash-Moser theorem can be applied and $\mathcal F$ is locally invertible with smooth tame inverse. As a direct consequence, arguing as in Proposition 5.3 in \cite{Advances} we have the following  

\medskip 
{\bf Claim 1.}
{\em For all $\epsilon,T>0$, there exists $\delta'>0$ such that if $f_0\in C^{\infty}(M,U)$ satisfies 
$$
\|f_0-\bar f\|_{C^\infty}<\delta'\,,
$$
then \eqref{flow_Ham} has a solution $f\in C^{\infty}(M\times[0,T],U)$ and 
$$
\|f-\bar f\|_{n,[0,T]}<\epsilon 
$$ 
for all $n \in \mathbb N$.}

\medskip 
Then we show the following

\medskip
{\bf Claim 2.}
{\em For $\delta'$ small enough the $L^2$-norm of $E(f)$ with respect to $\bar h$ has an exponential decay.} 

\medskip 
Fix a small time $\tau > 0$ arbitrary. By Claim 1 we have that there exists $\delta' >0$ such
that if $\|f_0-\bar f\|_{C^\infty}<\delta'$, then problem \eqref{flow_Ham} has a solution $f \in C^\infty(M \times [0, T + 2\tau], U)$ with $\|f - \bar f\|_{l, [0, T+2\tau]}$
bounded for every $l$. 
Now since
$$
\frac{\partial}{\partial t} E(f) = DE(f)(E(f))\,,
$$
we have
$$
\frac{d}{dt} \|E(f)\|^2_{L^2}=2 \langle \frac{\partial}{\partial t}E(f),E(f)\rangle_{L^2}=2 \langle DE(f)(E(f)),E(f)\rangle_{L^2}\,.
$$

Moreover a general result for families of symmetric operators on Hilbert spaces combined with the Sobolev Embedding theorem and elliptic regularity of $DE(\bar f)$ (see \cite{Advances} Corollary 5.6)
imply that given $a>0$ we can choose $\delta'$ so small that we have 
$$
\langle D E(f)(E(f)),E(f) \rangle_{L^2}\leq (1-a)\langle D E(\bar f)(E(f)),E(f) \rangle_{L^2}+a\|E(f)\|_{L^2}^2
$$
for every time in the interval $[0,T+\tau]$.
Now let $\lambda$ be half the smallest positive eigenvalue of $-DE(\bar f)_{| \ker L}$. 
Take $a=\frac{\lambda}{2\lambda+1}$ so that 
$$
\langle D E(f)(E(f)),E(f) \rangle_{L^2}\leq \frac{\lambda+1}{2\lambda+1}\langle D E(\bar f)(E(f)),E(f) \rangle_{L^2}+\frac{\lambda}{2\lambda+1}\|E(f)\|_{L^2}^2\,.
$$
By assumption \eqref{semidef} and \eqref{miracolo} we have $\langle D E(\bar f)(E(f)),E(f) \rangle_{L^2}\leq -2\lambda \|E(f)\|^2_{L^2} $ thus
$$
\langle D E(f)(E(f)),E(f) \rangle_{L^2}\leq- \lambda \|E(f)\|^2_{L^2}
$$
which implies 
$$
\frac{d}{dt}\|E(f)\|_{L^2}^2 \leq -2\lambda \|E(f)\|_{L^2}^2\,. 
$$
Using Gronwall's lemma we get  
$$
\|E(f(t))\|_{L^2}^2\leq e^{-2\lambda t} \|E(f_0)\|_{L^2}^2\,,\quad\mbox{ for all $t\in [0,T+\tau)$}\,,
$$
and the Claim 2 follows.

\medskip 
By integrating the last formula we get 
\begin{equation}
\label{exp2}
\|E(f)\|^2_{{0},[t,T+\tau]} = \int_{t}^{T+\tau}\|E(f(s))\|_{L^2}^2\,ds\leq \|E(f_0)\|^2_{L^2}\frac{e^{-2\lambda t}}{2\lambda}\,.
\end{equation}
Since using the parabolic Sobolev embedding theorem (see \cite{Advances} Corollary 5.8) 
there exist  $m$ and $C>0$ such that for every $t\in [0,T]$
$$
\|E(f(t))\|_{H^{n}} \leq C \|E(f)\|_{{m},[t,T+\tau]}\,,
$$
we will need the following estimate in order to prove $H^n$-exponential decay.

\medskip 
{\bf Claim 3.} {\em Let $\tau_0\in (0,T)$. For $\delta'$ small enough we have that for every 
 $m\in \mathbb N$ there exists $C>0$  such that
$$
\|E(f)\|_{{m},[t,T+\tau]} \leq C\|E(f)\|_{{0},[t-\tau_0,T+\tau]}\,,
$$
for every $t\in [\tau_0,T]\,.$}

\medskip 
We prove by induction on $m$ that for every $\tau_0 \in (0,T)$ we can choose $\delta'$ small enough such that there exists  a positive $C$ (depending on  $m$, 
$\tau_0$ and an upper bound on $\delta'$) such that for every $g\in C^{\infty}(M\times [0,T+\tau], F)$ solving  
$$
\frac{\partial}{\partial t}g=DE(f)(g)\,,
$$
the following estimate holds
\begin{equation}
\label{est1}
\|g\|_{{m},[t,T+\tau]} \leq C\|g\|_{{0},[t-\tau_0,T+\tau]}\,,
\end{equation}
for every $t\in [\tau_0,T]$.
Then we deduce the claim by setting $g=E(f)$. \\
For $m=0$ the estimate \eqref{est1} is trivial. We assume the above statement true up to $m=N$. For a smooth family of linear second order differential operator $P$ we set 
$$
|[P]|_{N} =\sum_{2j\leq N}[\tfrac{\partial^j}{\partial t^j}P]_{N-2j}
$$
where $[P]_N$ is the supremum of the norm of $P$ and its space covariant derivatives up to degree $N$. By [2, Lemma 6.10]
for $\delta'$ small enough there exists $C >0$,
depending on
$T$ and an upper bound on $\delta'$, such that for $t\in  [0, T+\tau)$ and $g\in C^{\infty}(M\times [0,T+\tau], F)$ we have
\begin{multline*}
\|g\|_{N+2,[t,T+\tau]} \leq C\left(\|\tfrac{\partial}{\partial t}g-DE(f)(g)\|_{N,[t,T+\tau]}+\|g(t)\|_{H^{N+1}}\right)\\
+C|[DE(f)]|_{N}\left(\|\tfrac{\partial}{\partial t}g-DE(f)(g)\|_{0,[t,T+\tau]}+\|g(t)\|_{H^{1}}\right)\,,
\end{multline*}
which implies 
$$
\|g\|_{N+2,[t,T+\tau]} \leq C(1+|[DE(f)]|_{N}) \left(\|\tfrac{\partial}{\partial t}g-DE(f)(g)\|_{N,[t,T+\tau]}+\|g(t)\|_{H^{N+1}}\right)\,.
$$
Up to shrink $\delta'$ we may assume 
$$
 |[DE(f)]|_{N}\leq  1+ |[DE(\bar f)]|_N
$$
in order to rewrite the last estimate as
\begin{multline}
\label{est2}
\|g\|_{N+2,[t,T+\tau]} \\
\leq C(2+|[DE(\bar f)]|_{N}) \left(\|\tfrac{\partial}{\partial t}g-DE(f)(g)\|_{N,[t,T+\tau]}+\|g(t)\|_{H^{N+1}}\right)\,.
\end{multline}
Now assume that $g$ satisfies 
$$
\frac{\partial}{\partial t}g=DE(f)(g)
$$
and let $\chi\colon \R\to [0,1]$ be smooth and such that 
$$
\begin{cases}
\chi(s)=0	\quad \mbox{ for } s \leq t-\frac{\tau_0}{2}\\
\chi(s)=1    \quad \mbox{ for } s \geq t\,.
\end{cases}
$$
Then $\tilde g=\chi g$ satsfies 
$$
\frac{\partial}{\partial t}\tilde g=DE(f)(\tilde g)+\dot\chi g
$$
and from \eqref{est2} we deduce
\begin{multline*}
\|g\|_{N+2,[t,T+\tau]}\leq \|\tilde g\|_{N+2,[t-\frac{\tau_{0}}{2},T+\tau]} \\
\leq C(2+|[DE(\bar f)]|_{N}) \|\dot\chi g\|_{N,[t-\frac{\tau_0}{2},T+\tau]}
\leq C' \|g\|_{N,[t-\frac{\tau_0}{2},T+\tau]}
\end{multline*}
and the claim follows using the inductive hypothesis.

\medskip 
Finally putting these together with \eqref{exp2} we have that for $\delta'$ small enough we have 
\begin{equation}
\label{Hndecay}
\|E(f(t))\|_{H^{n}} \leq C  \|E(f_0)\|_{L^2}{\rm e}^{{-\lambda} t}\,
\end{equation}
for $t \in [\tau_0,T]$. The constant $C$ may depend on $T,\tau_0$ and an upperbound on $\delta'$, but not on  
$t$ and $f_0$. 

\medskip 
Now we choose $\delta \leq \delta'$ such that  if $\|f_0-\bar f\|_{C^{\infty}}\leq \delta$ then
\begin{equation}\label{condizione}
C\|E(f_0)\|_{L^2}\frac{{\rm e}^{-\lambda \tau_0}}{\lambda}\sum_{j=0}^{\infty}{\rm e}^{-\lambda j(T-\tau_0)}+\|f({\tau_0})-\bar f\|_{H^{n}}\leq \delta'\,. 
\end{equation}

Using \eqref{Hndecay} and \eqref{condizione} and working as in \cite{Advances} we have that for  any $t\in [NT-(N-1)\tau_0,(N+1)T-N\tau_0]$ with $N \in \mathbb N$ 
$$
\|f-\bar f\|_{H^{n}}\leq C\|E(f_0)\|_{L^2}\frac{{\rm e}^{-\lambda \tau_0}}{\lambda}\sum_{j=0}^{N}{\rm e}^{-\lambda j(T-\tau_0)}+\| f({\tau_0})-\bar f\|_{H^{n}}\leq \delta'.
$$
This allows us to conclude that the solution $f$ is defined in $M\times [0,\infty)$. 

Now let $f_\infty=f_0+\int_{0}^\infty E(f)ds\in  C^{\infty}(M,F)$;
since 
$$
\lim_{t\to \infty}\|f(t)-f_{\infty}\|_{H^n}\leq \lim_{t\to \infty}C\|E(f_0)\|_{L^2}{\rm e}^{-\lambda t}=0\,, \mbox{ for $n$ large enough} 
$$
$f(t)$ converges to $f_\infty$ in $C^\infty$-topology. Possibly shrinking $\delta$, we will have $f_\infty\in C^{\infty}(M,U)$. 
Using again \cite[Proposition 5.7]{Advances}, we have that up to shrink $\delta$, 
$$
\|f(t)-\bar f\|_{C^{\infty}}<\epsilon
$$
for every $t\in [0,\infty)$. 

Finally 
$$
E(f_\infty)=\lim_{t\to 0} E(f(t))=0
$$
and the claim follows. 
\end{proof}

\section*{Acknowledgements} The authors would like to thank the anonymous referee for pointing out several inaccuracies and helping to considerably improve the presentation of the paper.

%\bibliography{bibliography}
%\bibliographystyle{mrl}

\end{document}